\documentclass[12pt]{amsart}
\usepackage{amsmath,amssymb,mathtools,tikz-cd}
\newcommand{\II}{\mathbb{I}}
\newcommand{\IE}{\mathbb{E}}

\newcommand{\IN}{\mathbb{N}}
\newcommand{\IR}{\mathbb{R}}
\newcommand{\IC}{\mathbb{C}}
\newcommand{\IB}{\mathbb{B}}
\newcommand{\IK}{\mathbb{K}}
\newcommand{\IM}{\mathbb{M}}
\newcommand{\IP}{\mathbb{P}}
\newcommand{\cH}{\mathcal{H}}

\newcommand{\ve}{\varepsilon}
\newcommand{\vp}{\varphi}
\newcommand{\acts}{\curvearrowright}
\newcommand{\Iin}{\mathbf{1}}
\newcommand{\Univ}{\mathrm{Univ}}
\newcommand{\Bad}{\mathrm{Bad}}

\newcommand{\op}{\mathrm{op}}
\newcommand{\Grass}{\mathrm{Grass}}
\newcommand{\Isom}{\mathrm{Isom}}
\newcommand{\Cstar}{\mathrm{C}^*\!}

\DeclareMathOperator{\Ad}{Ad}
\renewcommand{\Im}{\operatorname{Im}}
\DeclareMathOperator{\supp}{supp}
\DeclareMathOperator{\Ball}{Ball}
\DeclareMathOperator{\diag}{diag}
\DeclareMathOperator{\Width}{Width}
\DeclareMathOperator{\Roe}{\mathrm{C}^*_{\mathrm{u}}\!}
\DeclareMathOperator{\Ql}{\mathrm{C}^*_{\mathrm{ql}}\!}
\DeclarePairedDelimiter{\abs}{\lvert}{\rvert}
\DeclarePairedDelimiter{\norm}{\lVert}{\rVert}
\DeclarePairedDelimiter{\inpro}{\langle}{\rangle}
\DeclarePairedDelimiter{\paren}{(}{)}

\DeclarePairedDelimiter{\set}{\lbrace}{\rbrace}
\newcommand{\setcolon}{\mathrel{:}}

\newcommand{\pmat}[1]{%
  \begin{pmatrix}
    #1
  \end{pmatrix}%
}
\newcommand{\Vmat}[1]{%
  \begin{Vmatrix}
    #1
  \end{Vmatrix}%
}
\newcommand{\smallmat}[1]{%
  \paren{\begin{smallmatrix}
    #1
  \end{smallmatrix}}%
}
\newtheorem{thmA}{Theorem}

\newtheorem{corA}[thmA]{Corollary}
\newtheorem{thm}{Theorem}
\newtheorem{prop}[thm]{Proposition}
\newtheorem{lem}[thm]{Lemma}
\theoremstyle{definition}
\newtheorem{defn}[thm]{Definition}
\theoremstyle{remark}
\newtheorem{rem}[thm]{Remark}

\title[Embeddings into the universal Roe algebra]{Embeddings of matrix algebras into the universal Roe algebra}
\author{Hiroto Nishikawa}
\address{RIMS, Kyoto University, \mbox{606-8502} Japan}
\email{hnishika@kurims.kyoto-u.ac.jp}
\thanks{The author was partially supported by JSPS KAKENHI Grant Number 26KJ1456}

\date{\today}

\begin{document}
\begin{abstract}
    V. Manuilov introduced the universal Roe algebra, namely the uniform Roe algebra associated with the largest uniformly locally finite coarse structure $\Univ$ on $\IN$, together with several variants. 
    We study these algebras through embeddability properties of $\Cstar$-algebras.
    
    Among other results, we show that the universal Roe algebra admits an embedding from $\prod_n\IM_n$, answering a question by Farah.
    Thus, Ozawa's non-embeddability theorem for $\prod_n\IM_n$ in uniform Roe algebras of uniformly locally finite metric spaces does not extend to arbitrary uniformly locally finite coarse spaces.
    Nevertheless, embeddability properties distinguish the variants introduced by Manuilov and the quasi-local algebra of $\Univ$, apart from the multiplier algebra $M(\II)$.
    As a by-product, we also obtain an operator-norm analogue of Birkhoff's Problem 111: the norm-closed convex hull of permutation operators is exactly the set of doubly substochastic operators.
    
    Moreover, we show that every separable $\Cstar$-algebra embeds into the quasi-local algebra of $\Univ$. Although we do not know whether the analogous statement holds for the universal Roe algebra, we prove that every separable quasidiagonal $\Cstar$-algebra embeds into it, and hence into the reduced $\Cstar$-algebra of some \'{e}tale locally compact second countable groupoid.
\end{abstract}
\maketitle
\section{Introduction}
Roe algebras are important $\Cstar$-algebras which encode the large-scale geometry of metric spaces \cite{Roe88}.
For a uniformly locally finite metric space $X$, the uniform Roe algebra $\Roe X$ is the norm closure of finite-propagation operators, that is, operators on $\ell_2X$ with nonzero entries only within a bounded-width tube along the diagonal.
Membership in a uniform Roe algebra is often easy to verify by explicitly constructing finite-propagation approximations.
By contrast, proving non-membership is generally harder, since one must obtain a uniform lower bound on the distance from all finite-propagation operators.
On the other hand, it is easy to say that $T$ is not quasi-local: equivalently, there exist $\ve>0$ and subsets $A_n,B_n\subset X$ such that $\mathrm{dist}(A_n,B_n)>n$ and $\norm{\Iin_{A_n}T\Iin_{B_n}}\geq\ve$ for all $n$.
Since every uniform Roe operator is quasi-local, this also implies that $T\notin\Roe X$.
Therefore, the converse (i.e., whether the $\Cstar$-algebra of quasi-local operators coincides with the uniform Roe algebra) has attracted considerable attention \cite{Eng18, SZ20}.
This converse holds for spaces with property A, but fails in general.
The first example, where the uniform algebra and the quasi-local algebra do not coincide, is given in \cite{Oza25}.
Ozawa distinguished the two algebras by the embeddability of $\prod_n\IM_n$: he proved that $\prod_n\IM_n$ does not embed into the uniform Roe algebra of any uniformly locally finite metric space, whereas it embeds into the quasi-local algebra of a space containing expanders.

In this paper, we study similar problems for the universal ulf coarse structure $\Univ$. 
To be more precise, let us explain some notations.
Let $X$ be a countable set.

\begin{defn}
    For $\Sigma\subset X\times X$, denote its \emph{width} by $\Width\Sigma$
    \[\Width\Sigma\coloneqq\max\paren[\big]{\sup_{x}\abs{\set{y\setcolon(x,y)\in\Sigma}},\sup_y\abs{\set{x\setcolon(x,y)\in\Sigma}}}.\]
    The universal ulf coarse structure $\Univ$ is the set of $\Sigma\subset\IN\times\IN$ with bounded width.
\end{defn}

\begin{defn}
    \begin{itemize}
        \item A vector $\xi\in\ell_2X$ is said to be \emph{$d$-sparse} if its support $\supp(\xi)=\set{x\mid\xi_x\neq0}$ has the cardinality at most $d$.
        $\xi\in\ell_2X$ is \emph{$\ve$-$d$-sparse} if there exists a $d$-sparse vector $\xi'$ such that $\norm{\xi-\xi'}\leq\ve$.
        \item An operator $T\in\IB(\ell_2X)$ is said to be \emph{$d$-width} if its \emph{support} $\supp(T)=\set{(x,y)\mid T_{x,y}\neq0}$ has the width at most $d>0$.
        $T\in\IB(\ell_2X)$ is \emph{$\ve$-$d$-width} if there exists a $d$-width operator $T'$ such that $\norm{T-T'}\leq\ve$.
        \item An operator $T\in\IB(\ell_2X)$ is said to be \emph{$\ve$-$d$-quasi-local} if there exists $\Sigma\subset X\times X$ with $\Width\Sigma\leq d$ such that for every $A,B\subset X$ with $(A\times B)\cap\Sigma=\emptyset$ it satisfies $\norm{\Iin_AT\Iin_B}\leq\ve$.
        \item An operator $T\in\IB(\ell_2X)$ is said to be \emph{$d$-$\ell_1$bounded} if the operator norms on $\ell_1,\ell_\infty$ satisfy $\sum_x\abs{T_{x,y}},\sum_x\abs{T_{y,x}}\leq d$ for $y\in X$.
        \item An operator $T\in\IB(\ell_2X)$ is said to be \emph{$\ve$-$d$-inQ} if all columns and rows $T\delta_x,T^*\delta_x$ are $\ve$-$d$-sparse for $x\in X$.
        \item An operator $T\in\IB(\ell_2)$ is said to be \emph{ghost} if the matrix entries $T_{x,y}\to0\text{ as }x,y\to\infty$, that is, entries belong to $c_0(\IN\times \IN)$.
    \end{itemize}
\end{defn}

Then, we define the universal Roe algebra and its variants introduced in \cite{Man19}.

\begin{defn}
    \begin{itemize}
        \item The universal Roe algebra (called $\IB_f$ in \cite{Man19}) is denoted by
        \[\Roe\Univ\coloneqq\set{T\in\IB(\ell_2)\mid\forall\ve>0\exists d>0\ T\text{ is $\ve$-$d$-width}}.\]
        \item The quasi-local algebra $\Ql\Univ$ is defined as
        \[\Ql\Univ\coloneqq\set{T\in\IB(\ell_2)\mid\forall\ve>0\exists d>0\ T\text{ is $\ve$-$d$-quasi-local}}.\]
        \item Set $\IB_T\coloneqq\overline{\set{T\in\IB(\ell_2)\mid\exists d>0\ T\text{ is $d$-$\ell_1$bounded}}}$.
        \item $Q\coloneqq\set{T\in\IB(\ell_2)\mid\forall\ve>0\exists d>0\ T\text{ is $\ve$-$d$-inQ}}$ coincides with $\IB_\omega$ in \cite{Man19} by \cite{GQW24}.
        \item Since $\II\coloneqq\set{T\in\Roe\Univ\mid T\text{ is ghost}}$ is a nonunital $\Cstar$-algebra, we consider its multiplier algebra $M(\II)$.
    \end{itemize}
\end{defn}

Here is our main theorem.

\begin{thmA}
    $\prod_n\IM_n$ embeds into $\Roe\Univ$.
    
    Indeed, for every $\ve>0$, there exist $d>0$ and $C>0$, depending only on $\ve$, such that
    for any $n\in\IN$ there exists $N\in\IN$ satisfying the following:
    for a uniformly random multiplicity-one $*$-homomorphism $\vp\colon\IM_n\to\IM_N$ with probability at least $1-\exp\paren[\big]{-N/\exp(C n)}$,
    the operator $\vp(a)$ is $\ve$-$d$-width for any contraction $a\in\IM_n$.
\end{thmA}

The amplified matrix size can be taken to be $N=\exp(O_\ve(n))$.

\begin{thmA}
    $\Roe\Univ=\IB_T\subsetneq\Ql\Univ\subsetneq Q$ and $M(\II)\subset Q$ hold.

    Moreover, we give an operator-norm analogue of Birkhoff's Problem 111: the norm-closed convex hull of permutation operators is exactly the set of doubly substochastic operators on $\ell_2$.
\end{thmA}


\begin{thmA}
    There exists a bounded degree graph $X$ such that every separable $\Cstar$-algebra embeds into $\Ql X$ but $\IB(\ell_2)$ does not. 
    
    Every separable quasidiagonal $\Cstar$-algebra embeds into $\Roe\Univ$.
    Hence, for any such algebra $A$ there exists \'{e}tale locally compact second countable groupoid $G$ such that $A$ embeds into $\mathrm{C}^*_{\text{r}}G$.
\end{thmA}

\subsection*{Acknowledgements}
The author thanks his supervisor, Narutaka Ozawa, for his support and helpful comments, especially for bringing Farah's question to the author's attention and for his comments on Birkhoff's Problem 111.
This research was partially supported by JSPS KAKENHI Grant Number 26KJ1456.

\subsection*{AI disclosure}
The author used ChatGPT (OpenAI) during this study for proving Proposition~\ref{AI1} ($\Roe\Univ=\IB_T$, Birkhoff's problem 111), and for language editing.
All mathematical arguments were independently checked by the author, who takes full responsibility for the contents of the paper.

\subsection*{Outline of the proofs}
The proof of Theorem A is based on a degree-reduction using the zig-zag product and a discretization of the Ornstein--Uhlenbeck semigroup, viewed as a Gaussian analogue of linearity testing;
we explain this in detail in the next section.

For Theorem B, we use a matrix discrepancy theorem of Kyng--Luh--Song, random projection and a tensor-product inflation in order to distinguish the variants of $\Roe\Univ$.

Finally, Theorem C combines Theorems A and B with a block-tridiagonalization argument going back to Voiculescu.

\subsection*{Notations}
Denote the set $\set{1,\dots,N}$ by $[N]$.
$\Iin$ stands for the constant function $1$ and $\Iin_E$ stands for the indicator function on $E$.
A set $F$ in a metric space $X$ is said to be \emph{$\ve$-dense} if for every $x\in X$ there is $y\in F$ such that $d(x,y)\leq\ve$.
$F\subset X$ is said to be \emph{$\ve$-sparse} if $d(x,y)\geq\ve$ whenever $x,y\in F$ are distinct.

\clearpage
\section{Proof of Theorem A}
The proof of Theorem~A uses a degree reduction via the \emph{zig-zag product} and \emph{linearity testing}, as in \cite{Din07}.
The author thanks Nobutaka Shimizu for his clear lectures on the PCP theorem, through which the author learned these methods.

We first explain the main ideas of the proof.
First, we may assume that the contraction $a\in\IM_n$ is a projection by the layer-cake decomposition. 
It therefore remains to show that the Haar random projection with small rank is $\ve$-$d$-width with high probability.
Note that the projection $\frac1N\Iin\otimes\Iin^*$ onto the constant vectors $\IC\Iin\subset\ell_2^N$ is $\ve$-$d$-width for $\ve=O(d^{-\frac12})$ thanks to expanders.
Our proof uses the Gaussian-weighted version of this argument, explained below.

We work over the real field; the complex case will be recovered by a $2\times2$ trick. 
Let $V$ be the $N\times n$ matrix whose entries are i.i.d. Gaussian with variance $\frac1N$. 
$VV^*$ is close to the Haar random projection in the operator norm, so we show that $VV^*$ is $\ve$-$d$-width.

Viewing $V$ as the map $F\colon[N]\to\IR^n$ mapping each point to the row of $V$, we divide $X\coloneqq[N]$ into $\exp(O(n))$ parts $\{X_i\}_{i\in I}$ corresponding to a partition $\IR^n=\bigsqcup_{i\in I} U_i$. 
Since $F$ is almost constant on almost all $X_i$, we choose an expander on each $X_i$. 
Define the weighted graph on $I$ coming from the linearity testing on $\IR^n$ (so-called \emph{Ornstein--Uhlenbeck semigroup}).
By taking the zig-zag product with $I$ and $\{X_i\}$, we can reduce the degree of the linearity testing, and $VV^*$ is approximated by some spectral projection of the Laplacian of the zig-zag product.
In this way, we compare the Gaussian space $\IR^n$ with $[N]$:
\[\begin{tikzcd}
    \IR^n\ar[r,"V_\infty"]\ar[d,phantom,"=" sloped] 
    &L^2(\IR^n,\gamma)\ar[r,"P_\infty"] \ar[d,phantom,"\approx" sloped]
    &\ell_2(I,\mu_\infty)\ar[d,"\iota"]\\
    \IR^n\ar[r,"V"] &\ell_2[N]\ar[r,"P"] &\ell_2(I,\mu)
\end{tikzcd}\]

\subsection{Reduction to real, projection, Gaussian}

\begin{prop}\label{Prop}
    Let $V\in\IM(N\times n,\IR)$ be the random matrix whose entries are i.i.d. Gaussian with variance $\frac1N$.
    For every $\ve>0$, there exist $d>0$ and $C>0$, depending only on $\ve$, such that for every $n$ there exists $N_0$ satisfying the following:
    for every $N\geq N_0$, $VV^*$ is $\ve$-$d$-width with probability at least $1-\exp\paren[\big]{-N\exp(-Cn)}$.
\end{prop}

We show that this proposition implies Theorem~A in this subsection, and prove the proposition in the next subsection.

\begin{proof}[Proof of Prop \ref{Prop}$\Rightarrow$Thm A]
    First, we show that Proposition~\ref{Prop} implies
    the corresponding complex statement by the $2\times2$ trick.
    Let $V'$ be an independent copy of $V$, then $\frac1{\sqrt{2}}(V+iV')\in\IM(N\times n,\IC)$ is the complex Gaussian matrix.
    Applying Proposition~\ref{Prop} to $\frac{1}{\sqrt{2}}\smallmat{V\\V'}\in\IM(2N\times n,\IR)$, we obtain that $\frac12\smallmat{V\\V'}\smallmat{V\\V'}^*$ is $\ve$-$d$-width with high probability.
    Therefore, $\frac12 VV^*,\frac12VV'^*,\frac12V'V^*,\frac12V'V'^*$ are $\ve$-$d$-width and $\frac12(V+iV')(V+iV')^*$ is $4\ve$-$4d$-width with high probability.

    Second, we can replace the complex Gaussian matrix $V\in\IM(N\times n,\IC)$ with the Haar isometry $V'\in\Isom(N\times n)$.
    Consider the polar decomposition of $V\in\IM(N\times n)$ $V=V'\abs{V}$ for some $V'\in\Isom(N\times n)$ with $\Im V'=\Im V$.
    By the rotation invariance of the Gaussian, $V'\in\Isom(N\times n)$ is distributed uniformly, so it suffices to show that $\norm{V'-V}\leq\ve$ with high probability.
    This follows from $1-\ve\leq\abs{V}\leq1+\ve$ \cite[Thm. 4.6.1]{Ver18}.
    Thus, the Haar random projection with rank $n$ is $\ve$-$d$-width with probability at least $1-\exp(-N\exp(-O_\ve(n)))$.

    Third, we show that $VaV^*$ is $\ve$-$d$-width for every $a\in\IM_n^{\leq1}$ and the Haar random $V\in\Isom(N\times n)$ with high probability.
    Choose an $\ve$-dense subset $F\subset\set{p\in\IM_n\mid p\text{ is projection}}$ with $\abs{F}\leq (1+\frac2\ve)^{2n^2}$ \cite[Cor. 4.2.11]{Ver18}.
    Since $VpV^*$ is distributed as the Haar random projection with rank $\mathrm{rank}(p)$ for $p\in F$, 
    \[\IP[\forall p\in F\ VpV^*\text{ is $\ve$-$d$-width}]\geq1-\abs{F}\exp(-N\exp(-O_\ve(n)))\]
    holds.
    We claim that if $VpV^*$ is $\ve$-$d$-width for every $p\in F$ then $VaV^*$ is $O(\ve)$-$dO(\ve^{-1})$-width for every $a\in\IM_n^{\leq1}$.
    This claim finishes the proof because $\log\abs{F}=O(n^2)$ is sufficiently smaller than $N\exp(-O_\ve(n))$.

    We show the claim.
    $VpV^*$ is $2\ve$-$d$-width for every projection $p\in\IM_n$.
    Every positive contraction $a\in\IM_n^{+,\leq1}$ can be approximated within $\ve$ by a convex combination of at most $1+\frac{1}{\ve}$ projections by the layer-cake decomposition, so $VaV^*$ is $3\ve$-$(1+\frac{1}{\ve})d$-width.
    Therefore, $VaV^*$ is $12\ve$-$4(1+\frac1\ve)d$-width for every contraction $a\in\IM_{n}^{\leq1}$.
\end{proof}

\subsection{Proof of Prop \ref{Prop}}

Throughout this section, we work over the real field.
Set the Gaussian random matrix $V\in\IM(N\times n,\IR)$ and the embedding $V_\infty\colon\IR^n\hookrightarrow L^2(\IR^n,\gamma)$ mapping to linear functions on $\IR^n$.
To compare these (almost-)isometries, fix a partition $\IR^n=\bigsqcup_{i\in I}U_i$ and $[N]=\bigsqcup_{i\in I}X_i$, where $X_i\coloneqq F^{-1}(U_i)$ and $F\colon[N]\to\IR^n$ maps each point $p$ to the $p$-th row of $\sqrt{N}V$.
Define probability vectors $\mu_\infty\colon I\ni i\mapsto\gamma(U_i)\in\IR$ and $\mu\colon I\ni i\mapsto\frac{\abs{X_i}}{N}\in\IR$, and the projections $P_\infty$ onto $\ell_2(I,\mu_\infty)\subset L^2(\IR^n,\gamma)$ and $P$ onto $\ell_2(I,\mu)\subset\ell_2^N$ which correspond to the partitions.
The identification $\iota\colon\ell_2(I,\mu_\infty)=\IC^I=\ell_2(I,\mu)$ has the operator norm $\max_i\sqrt{\frac{\mu(i)}{\mu_\infty(i)}}$ and is bounded below with $\min_i\sqrt{\frac{\mu(i)}{\mu_\infty(i)}}$.

To avoid confusion, we only treat probability measures and $L^2$-spaces on these, except $\IR^n$.
In this convention, we denote $\sqrt{\frac1N\sum_{l\in[N]}\abs{\xi(l)}^2}$ for $\xi\in\ell_2$ by $\norm{\xi}$ because this is the $L^2$-norm of the uniform probability measure on $[N]$.
$V_\infty(y)=(x\mapsto\inpro{y,x})\in L^2(\IR^n,\gamma)$, $V(y)=(p\mapsto\inpro{y,F(p)})\in\ell_2^N$, $P_\infty(f)=(i\mapsto\mu_\infty(i)^{-1}\int_{U_i}fd\gamma)\in\ell_2(I,\mu_\infty)$ and $P(\xi)=(i\mapsto\mu(i)^{-1}\frac1N\sum_{l\in X_i}\xi(l))\in\ell_2(I,\mu)$ holds in our convention.
There is the natural unit vector $\Iin$ in every $L^2$-space, and $\Iin$ is preserved under $P, P_\infty,\iota$ but not $\iota^*$.
This will appear as the Perron--Frobenius eigenvector for many nonnegative matrices.

\begin{lem}
    Let $\ve>0$.
    There exists a partition $\{U_i\}_{i\in I}$ with $\abs{I}\leq O(\ve^{-1}\sqrt{\log\ve^{-1}})^n$ such that $\mu_\infty^{\text{min}}:=\min_i\mu_\infty(i)$ is $\exp(-O_\ve(n))$ and $\norm{(1-P_\infty)V_\infty}\leq\ve$.
\end{lem}

\begin{proof}
    First, we show this for $n=1$.
    Define $J_0:=\IR\setminus[-\ve m,\ve m]$ and let $J_1,\dots,J_{2m}$ be the intervals with length $\ve$ for some $m=O(\ve^{-1}\sqrt{-\log\ve})$.
    One has $\norm{(1-P_\infty)x}\leq\ve$ for the identity $x\in L^2(\IR,\gamma)$ by the Gaussian tail estimate.
    Actually, $O(\ve^{-1})$ suffices \cite[Thm. 6.2]{GL00}.
    
    Second, for general $n$, we define $I=\set{0,\dots,2m}^n$ and $U_i\coloneqq\prod_kJ_{i_k}\subset\IR^n$, where $i=(i_1,\dots,i_n)\in I$.
    Then $P_\infty$ for this partition on $L^2(\IR^n,\gamma)=L^2(\IR,\gamma_1)^{\otimes n}$ is $(P_\infty\text{ for $n=1$})^{\otimes n}$, and a linear function $\sum_k\lambda_kx_k\in L^2(\IR^n,\gamma)=L^2(\IR,\gamma_1)^{\otimes n}$ is $\sum_k\lambda_k(\Iin^{\otimes k-1}\otimes x\otimes\Iin^{\otimes n-k})$.
    Since $P_\infty\Iin=\Iin$, we have
    \[\norm{(1-P_\infty)\sum_k\lambda_k x_k}^2=\sum_k\abs{\lambda_k}^2\norm{(1-P_\infty)x}^2_{L^2(\IR)}\leq\ve^2\]
    for $\sum_k\abs{\lambda_k}^2\leq1$.
    We also have $\mu_\infty^{\text{min}}=(\min_l\gamma(J_l))^n\geq\exp(-O_\ve(n))$
\end{proof}

Note that this condition $\norm{(1-P_\infty)V_\infty}\leq\ve$ is equivalent to $\norm{P_\infty V_\infty(y)}\geq\sqrt{1-\ve^2}\norm{y}$ for every $y\in\IR^n$ because $\abs{P_\infty V_\infty}^2+\abs{(1-P_\infty)V_\infty}^2=1$.

Since $\{F(l)\}_{l\in[N]}$ are independent random variables distributed on $(\IR^n,\gamma)$,
\[\mu(i)=\frac1N\sum_l\Iin_{U_i}(F(l))\to\IE_{X\sim\gamma}[\Iin_{U_i}(X)]=\mu_\infty(i)\quad\text{as }N\to\infty\]
by the law of large numbers.
If we need the quantitative behavior, we apply the multiplicative Chernoff bound to deduce $\IP[e^{-\delta}\mu_\infty(i)\leq\mu(i)\leq e^\delta\mu_\infty(i)]\geq1-2\exp(-N\mu_\infty^{\text{min}}\frac{\delta^2}{3})$
for $\delta>0$.
Thus, $e^{-2\delta}\leq\iota^*\iota\leq e^{2\delta}$ with high probability.
For $y\in\IR^n$ with $\norm{y}=1$, we have
\[\mu(i)(PV(y))_i=\frac1N\sum_l\Iin_{U_i}(F(l))\inpro{y,F(l)}\to\IE_{X\sim\gamma}[\Iin_{U_i}(X)\inpro{y,X}]=\mu_\infty(i)(P_\infty V_\infty(y))_i.\]
We also have the quantitative estimate.

\begin{lem}
    $\norm{PV-\iota P_\infty V_\infty}\leq O(\delta)$ holds with high probability.
\end{lem}

\begin{proof}
    Set $\vp(t)\coloneqq\exp(t)-1-t,\ \ve_i\coloneqq\delta\mu_\infty(i)+\delta^{-1}\max_{\pm}\int_{U_i}\vp(\pm\delta\inpro{y,x})d\gamma(x)$.
    
    First, we show that $\IP[\abs{\mu(i)PV(y)_i-\mu_\infty(i)P_\infty V_\infty(y)_i}\geq\ve_i]\leq2\exp(-N\mu_\infty(i)\delta^2)$.
    By the Chernoff bound and $\log\IE[\exp(\pm\delta\Iin_{U_i}(X)\inpro{y,X})]\leq\IE[\Iin_{U_i}(X)(\exp(\pm\delta\inpro{y,X})-1)]$, it suffices to show that $\ve_i\delta-\IE[\Iin_{U_i}\vp(\pm\delta\inpro{y,X})]\geq\delta^2\mu_\infty(i)$, which is trivial.

    Second, we show that $\sum_i\frac{\ve_i^2}{\mu_\infty(i)}$ is small. Indeed, one has
    \begin{align*}
        \sqrt{\sum_i\frac{\ve_i^2}{\mu_\infty(i)}}&\leq\sqrt{\sum_i\frac{(\delta\mu_\infty(i))^2}{\mu_\infty(i)}}+\sqrt{\sum_i\frac{\bigl(\delta^{-1}\max_\pm\int_{U_i}\vp\dots\bigr)^2}{\mu_\infty(i)}}\\
        &=\delta+\delta^{-1}\sqrt{\sum_i\max_\pm\frac{\bigl(\int_{U_i}\vp(\pm\delta\inpro{y,x})d\gamma(x)\bigr)^2}{\mu_\infty(i)}}\\
        &\leq\delta+\delta^{-1}\sqrt{\sum_i\max_\pm\int_{U_i}\vp(\pm\delta\inpro{y,x})^2d\gamma(x)}\\
        &\leq\delta+\delta^{-1}\sqrt{\int_\IR\bigl(\vp(-\delta x)^2+\vp(\delta x)^2\bigr)d\gamma_1(x)}\\
        &=O(\delta).
    \end{align*}
    
    Third, we show that $\norm{PV(y)-\iota P_\infty V_\infty(y)}\leq O(\delta)$ for $y\in F$, where $F\subset S^{n-1}$ is a $\frac12$-dense subset with $\abs{F}\leq5^n$.
    Define $\xi_i\coloneqq PV(y)_i$ and $\xi'_i\coloneqq P_\infty V_\infty(y)_i$.
    Except with probability $2\abs{I}(1+\abs{F})\exp(-N\mu_\infty^{\text{min}}\frac{\delta^2}{3})$, we may assume that $e^{-\delta}\leq\frac{\mu(i)}{\mu_\infty(i)}\leq e^\delta$ and $\abs{\mu(i)\xi_i-\mu_\infty(i)\xi'_i}\leq\ve_i$ for every $y\in F$.
    One has
    \begin{align*}
        \sqrt{\sum_i\mu(i)\abs{\xi_i-\xi'_i}^2}&\leq\sqrt{\sum_i\mu(i)\abs*{\xi_i-\frac{\mu_\infty(i)}{\mu(i)}\xi'_i}^2}+\sqrt{\sum_i\mu(i)\abs*{\left(1-\frac{\mu_\infty(i)}{\mu(i)}\right)\xi'_i}^2}\\
        &\leq e^{\frac\delta2}\sqrt{\sum_i\frac{\ve_i^2}{\mu_\infty(i)}}+e^{\frac\delta2}(e^\delta-1)\sqrt{\sum_i\mu_\infty(i)\abs{\xi'_i}^2}\\
        &\leq \sqrt{e^\delta}O(\delta)+\sqrt{e^\delta}(e^\delta-1)\norm{P_\infty V_\infty}\\
        &=O(\delta).
    \end{align*}
    Therefore, $\norm{PV-\iota P_\infty V_\infty}\leq2\sup_{y\in F}\norm{PV(y)-\iota P_\infty V_\infty(y)}\leq O(\delta)$ holds.
\end{proof}

We have $\norm{PV(y)}\geq (e^{-\delta}\sqrt{1-\ve^2}-O(\delta))\norm{y}\geq\sqrt{1+\delta-(\ve+O(\sqrt{\delta}))^2}\norm{y}$.
Hence, we also have $\norm{(1-P)V}\leq\ve+O(\sqrt\delta)$ and $\norm{V-\iota P_\infty V_\infty}\leq\ve+O(\sqrt{\delta})$ because $\abs{PV}^2+\abs{(1-P)V}^2=V^*V\leq1+\delta$ with high probability \cite[Thm. 4.6.1]{Ver18}.

The integral operator $T_t$ on $L^2(\IR^n,\gamma)$ defined as
\[(T_tf)(x)\coloneqq\int_{\IR^n}f(e^{-t}x+\sqrt{1-e^{-2t}}y)d\gamma(y)\]
with a positive symmetric kernel function is called the Ornstein--Uhlenbeck semigroup.
This operator has the spectrum $\{1,e^{-t},e^{-2t},\dots\}\cup\{0\}$ with the corresponding eigenspaces $\IR\Iin,\{\text{linear functions}\},\{\text{Hermite polynomials with degree 2}\},\dots$.
This may be viewed as a Gaussian analogue of linearity testing in \cite{BLR93}:
the quantity below controls the distance of $f\in L^2(\IR^n,\gamma)\ominus\IR\Iin$ from the space of linear functions.
\[\norm{f-e^tT_t(f)}^2\leq e^{2t}\IE_{X,Y\sim\gamma}[\abs{e^{-t}f(X)+\sqrt{1-e^{-2t}}f(Y)-f(e^{-t}X+\sqrt{1-e^{-2t}}Y)}^2].\]
By $\norm{T_t-(\Iin\otimes\Iin^*+e^{-t}V_\infty V_\infty^*)}\leq e^{-2t}$, the operator $T^I_t\coloneqq\iota P_\infty T_tP_\infty^*\iota^*$ on $\ell_2(I,\mu)$ satisfies $\norm{T^I_t-(\Iin\otimes\Iin^*+e^{-t}(\iota P_\infty V_\infty)(\iota P_\infty V_\infty)^*)}\leq e^{2\delta}e^{-2t}$ with high probability.
Thus, one also has $\norm{T^I_t-(\Iin\otimes\Iin^*+e^{-t}VV^*)}\leq e^{2\delta} e^{-2t}+e^{-t}(\ve+O(\sqrt\delta))$ with high probability.

Since $T^I_t\in\IB(\ell_2(I,\mu))$ has nonnegative entries and approximately fixes $\Iin$, we perturb it so that it fixes $\Iin$ exactly.

\begin{lem}
    There is $T'_t\in\IB(\ell_2(I,\mu))_{\text{sa}}$ such that each entry $\inpro{\delta_i,T'_t\delta_j}$ belongs to $N^{-1}\IN$, $T'_t\Iin=\Iin$ and $\norm{T^I_t-T'_t}\leq\frac{2e^\delta\abs{I}}{\mu_\infty^{\text{min}}N}+3\delta=3\delta+N^{-1}\exp(O(n))$.
\end{lem}

\begin{proof}
    Since $\iota^*\delta_i=\frac{\mu(i)}{\mu_\infty(i)}\delta_i$, entries $a(i,j)\coloneqq\inpro{\delta_i,T^I_t\delta_j}$ satisfy
    \[\sum_ja(i,j)=\sum_j\frac{\mu(i)\mu(j)}{\mu_\infty(i)\mu_\infty(j)}\inpro{\Iin_{U_i},T_t\Iin_{U_j}}\leq e^\delta\frac{\mu(i)}{\mu_\infty(i)}\inpro{\Iin_{U_i},T_t\sum_j\Iin_{U_j}}=e^\delta\mu(i).\]
    Similarly, we have $\sum_ja(i,j)\geq e^{-\delta}\mu(i)$.
    Define $n_{i,j}\coloneqq[Na(i,j)e^{-\delta}]\in\IN$ for $i\neq j$ and $n_{i,i}\coloneqq N\mu(i)-\sum_{j\neq i}n_{i,j}\in\IN$.
    Set $T'_t(i,j)\coloneqq\frac{n_{i,j}}{N}$.
    For a symmetric matrix $S\in\IB(\ell_2(I,\mu))$, the operator norms on $\ell_1(I,\mu)$ and $\ell_\infty(I,\mu)$ are both equal to $\sup_i\mu(i)^{-1}\sum_j\abs{S(i,j)}\geq\norm{S}_{\IB(\ell_2(\mu))}$ by the Schur test.
    Thus, we need to bound $\mu(i)^{-1}\sum_j\abs{a(i,j)-T'_t(i,j)}$.
    \begin{align*}
        \sum_j\abs{a(i,j)-T'_t(i,j)}
        \leq&\sum_ja(i,j)+\mu(i)-2\sum_jN^{-1}[Na(i,j)e^{-\delta}]\\
        \leq&\mu(i)+\frac{2\abs{I}}{N}-(2e^{-\delta}-1)\sum_ja(i,j)\\
        \leq&\mu(i)(1-2e^{-2\delta}+e^{-\delta})+\frac{2\abs{I}}{N}
    \end{align*}
    holds.
\end{proof}

Hence, $n_{i,j}=NT'_t(i,j)$ satisfies $\sum_jn_{i,j}=N\mu(i)=\abs{X_i}$ and $n_{i,j}=n_{j,i}$.
Decompose $X_i$ into $\bigsqcup_jX_{i,j}$ with $\abs{X_{i,j}}=n_{i,j}$.
Fix bijections from $X_{i,j}$ to $X_{j,i}$ for $i\neq j$ and combine them into an involution $\sigma$ on $\bigsqcup_{i,j} X_{i,j}=[N]$ by imposing on $\sigma$ to be identity on $X_{i,i}$.
Choose $d$-regular expanders $G_i$ on $X_i$ for an even integer $d>2$, and denote by $T_{G_i}$ the adjacency matrix, that is, the matrix whose $(x,y)$-entry is $\frac1d$ if $x,y$ are adjacent and 0 otherwise.
The zig-zag product is the $d^2$-regular graph whose adjacency matrix is $T\coloneqq(\bigoplus_iT_{G_i})\sigma(\bigoplus_iT_{G_i})$.

\begin{lem}
    $PTP=T'_t$ and $\norm{(1-P)T}\leq\max_ir(G_i)$ holds, where $r(G_i)\coloneqq\norm{T_{G_i}\lvert_{\IR\Iin^\perp}}$.
\end{lem}

\begin{proof}
    The former follows from $\inpro{\Iin_{X_i},\sigma\Iin_{X_j}}_{\ell_2^N}=\frac1N\abs{X_{i,j}}$.
    The latter follows from $(1-P)(\bigoplus T_{G_i})=\bigoplus((1-\Iin\otimes\Iin^*)T_{G_i})$.
\end{proof}

We may assume that $r(G_i)\leq\frac{2\sqrt{d-1}}{d}+o(1)$.
Hence, we have $\norm{T-T'_t}\leq O(d^{-\frac12})$.
Combining the preceding estimates, we obtain
\[\norm{T-(\Iin\otimes\Iin^*+e^{-t}VV^*)}\leq O(d^{-\frac12})+3\delta+N^{-1}\exp(O(n))+e^{2\delta}e^{-2t}+e^{-t}(\ve+O(\sqrt\delta)).\]
We choose $t,d,\delta$ so that $e^{-t}\lesssim\ve,d^{-\frac12}\lesssim\ve e^{-t}$ and $\delta\lesssim e^{-t}\ve\lesssim\ve^2$.
Then all right hand side except $N^{-1}\exp(O(n))$ are at most $e^{-t}O(\ve)$.
Since $N^{-1}\exp(O(n))\leq e^{-t}\ve$ for sufficiently large $N$, $\norm{VV^*-e^t(T-\Iin\otimes\Iin^*)}\leq O(\ve)$ holds.
By expanders, $e^t\Iin\otimes\Iin^*$ is $\ve$-$e^{2t}O(\ve^{-2})$-width, and $VV^*$ is $O(\ve)$-$d_0$-width, where $d_0=d^2+e^{2t}O(\ve^{-2})$, with probability at least $1-\exp(-N\exp(-O(n)))$.

\begin{proof}[Proof of Theorem A]
    So far, we have proved that $\forall\ve\exists d=d(\ve)\forall n\exists N_0\forall N\geq N_0$ the image of the random embedding $\vp:\IM_n\to\IM_N$ is $\ve$-$d$-width with probability greater than $1-\exp(-N\exp(-C_\ve n))$, for a constant $C_\ve$ depending on $\ve$.
    We construct $\Phi\colon\prod_n\IM_n\to\Roe\Univ$ by the random isometries $V_n\colon\ell_2^n\to\ell_2(X_n)$, where $\IN=\bigsqcup_nX_n$ is a partition with $\frac1n\log\abs{X_n}\to\infty$ as $n\to\infty$.
    Fix $\ve>0$. 
    One has
    \begin{align*}
        \sum_n\IP[\text{The image of the random embedding $\vp:\IM_n\to\IM_{\abs{X_n}}$ is not $\ve$-$d$-width}]\\
        \leq\sum_n\exp\paren[\big]{-\abs{X_n}\exp(-C_\ve n)}<\infty
    \end{align*}
    because there are only finitely many $n$ such that $\frac1n\log\abs{X_n}\leq C_\ve+1$.
    By the Borel--Cantelli lemma,
    with probability one, there exists $n_0$ such that, for every $n\geq n_0$, the whole unit ball of $\IM_n$ is mapped to $\ve$-$d$-width operators.
    Therefore,
    \[d'=\max\paren[\big]{d,\max_{n<n_0}\abs{X_n}}\]
    has the property that $\Phi(a)$ is $\ve$-$d'$-width for every $a\in\prod_n\IM_n^{\leq1}$ almost surely.
    Taking the intersection of the probability-one events for $\ve=2^{-m},\ m\in\IN$, finishes the proof.
\end{proof}
\clearpage

\section{Proof of Theorem B}

We review the quantitative notations and algebras in \cite{Man19}.

\begin{defn}
    \begin{itemize}
        \item A vector $\xi\in\ell_2X$ is said to be \emph{$d$-sparse} if its support $\supp \xi=\set{x\mid \xi_x\neq0}$ has the cardinality at most $d$.
        $\xi\in\ell_2X$ is \emph{$\ve$-$d$-sparse} if there exists a $d$-sparse vector $\xi'$ such that $\norm{\xi-\xi'}\leq\ve$.
        \item An operator $T\in\IB(\ell_2X)$ is said to be \emph{$d$-width} if its \emph{support} $\supp T=\set{(x,y)\mid T_{x,y}\neq0}$ has the width at most $d\geq0$.
        $T\in\IB(\ell_2X)$ is \emph{$\ve$-$d$-width} if there exists a $d$-width operator $T'$ such that $\norm{T-T'}\leq\ve$.
        \item An operator $T\in\IB(\ell_2X)$ is said to be \emph{$\ve$-$d$-quasi-local} if there exists $\Sigma\subset X\times X$ with $\Width\Sigma\leq d$ such that for every $A,B\subset X$ with $A\times B\cap\Sigma=\emptyset$ it satisfies $\norm{\Iin_AT\Iin_B}\leq\ve$.
        \item An operator $T\in\IB(\ell_2X)$ is said to be \emph{$d$-$\ell_1$bounded} if the operator norms on $\ell_1,\ell_\infty$ satisfy $\sum_x\abs{T_{x,y}},\sum_x\abs{T_{y,x}}\leq d$ for $y\in X$.
        \item An operator $T\in\IB(\ell_2X)$ is said to be \emph{$\ve$-$d$-inQ} if all columns and rows $T\delta_x,T^*\delta_x$ are $\ve$-$d$-sparse for $x\in X$.
        \item An operator $T\in\IB(\ell_2)$ is said to be \emph{ghost} if matrix entries $T_{x,y}\to0\text{ as }x,y\to\infty$, that is, entries belong to $c_0(\IN\times \IN)$.
    \end{itemize}
\end{defn}

\begin{defn}
    \begin{itemize}
        \item The universal Roe algebra (called $\IB_f$ in \cite{Man19}) is denoted by
        \[\Roe\Univ\coloneqq\set{T\in\IB(\ell_2)\mid\forall\ve>0\exists d>0\ T\text{ is $\ve$-$d$-width}}.\]
        \item The quasi-local algebra $\Ql\Univ$ is defined as
        \[\Ql\Univ\coloneqq\set{T\in\IB(\ell_2)\mid\forall\ve>0\exists d>0\ T\text{ is $\ve$-$d$-quasi-local}}.\]
        \item Set $\IB_T\coloneqq\overline{\set{T\in\IB(\ell_2)\mid\exists d>0\ T\text{ is $d$-$\ell_1$bounded}}}$.
        \item $Q\coloneqq\set{T\in\IB(\ell_2)\mid\forall\ve>0\exists d>0\ T\text{ is $\ve$-$d$-inQ}}$ coincides with $\IB_\omega$ in \cite{Man19} by \cite{GQW24}.
        \item Since $\II\coloneqq\set{T\in\Roe\Univ\mid T\text{ is ghost}}$ is a nonunital $\Cstar$-algebra, we consider its multiplier algebra $M(\II)$.
    \end{itemize}
\end{defn}

\subsection{Proof of $\Roe\Univ=\IB_T$}

The key ingredient is Kyng, Luh, and Song's theorem on matrix discrepancy \cite{KLS20}, which improves the celebrated solution to the Kadison--Singer problem by Marcus, Spielman, and Srivastava \cite{MSS15}.

\begin{prop}\label{AI1}
    Recall that a matrix $T\in\IM(n,\IR)$ is said to be \emph{doubly substochastic} if
    \[T_{i,j}\geq0,\quad \sum_iT_{i,j}\leq1\quad \text{and}\quad\sum_jT_{i,j}\leq1.\]
    There is a universal constant $C$ with the following property:
    For every $n\in\IN$, every doubly substochastic $T\in\IM(n,\IR)$, and every $\ve\in(0,1)$, $T$ is $\ve$-$\lfloor C\ve^{-2}\rfloor$-width.
\end{prop}

\begin{proof}
    Let $t>0$.
    Define independent random variables $\xi_{i,j}$ which take the value $\lfloor tT_{i,j}\rfloor+1$ with probability $tT_{i,j}-\lfloor tT_{i,j}\rfloor$ and take $\lfloor tT_{i,j}\rfloor$ with the remaining probability.
    Note that $\IE[\xi_{i,j}]=tT_{i,j}$ and $v_{i,j}\coloneqq\mathrm{Var}[\xi_{i,j}]\leq tT_{i,j}$.
    Set $u_{i,j}\coloneqq\delta_i\oplus\delta_j\in\ell_2^n\oplus\ell_2^n$.

    We apply \cite{KLS20} to the sum $\sum_{i,j}\xi_{i,j}u_{i,j}\otimes u_{i,j}^*\in\IB(\ell_2^n\oplus\ell_2^n)$ of rank one matrices. Then we have $\epsilon_{i,j}=\lfloor tT_{ij}\rfloor\text{ or }\lfloor tT_{ij}\rfloor+1$ such that 
    \[\norm*{\sum_{i,j}(tT_{i,j}-\epsilon_{i,j})u_{i,j}\otimes u_{i,j}^*}_{\op}\leq4\sigma,\quad\text{where }\sigma^2\coloneqq\norm*{\sum_{i,j}v_{i,j}(u_{i,j}\otimes u_{i,j}^*)^2}_{\op}.\]
    To bound $\sigma$, we apply the Schur test to the non-negative matrix $A\coloneqq\sum_{i,j}v_{i,j}(u_{i,j}\otimes u_{i,j}^*)^2$.
    Set the constant vector $\Iin\in\ell_2^n$.
    \[\norm{A}_{\ell_\infty\to\ell_\infty}
    =\norm{A(\Iin\oplus\Iin)}_{\ell_\infty}
    =\norm*{\sum_{i,j}4v_{i,j}u_{i,j}}_{\ell_\infty}
    \leq\Vmat{\sum_{i,j}4tT_{i,j}\delta_i\\\sum_{i,j}4tT_{i,j}\delta_j}_{\ell_\infty}
    \leq\Vmat{4t\Iin\\4t\Iin}_{\ell_\infty}
    =4t\]
    holds. 
    Similarly, one has $\norm{A}_{\ell_1\to\ell_1}\leq4t$ and $\norm{A}_{\ell_2\to\ell_2}=\sigma^2\leq4t$.

    Set $S\coloneqq(t^{-1}\epsilon_{i,j})_{i,j}\in\IM(n,t^{-1}\IN)$.
    One has
    \[\sum_{i,j}(tT_{i,j}-\epsilon_{i,j})u_{i,j}\otimes u_{i,j}^*
    =\sum_{i,j}t(T_{i,j}-S_{i,j})\pmat{e_{i,i}&e_{i,j}\\e_{j,i}&e_{j,j}}
    =\pmat{tD&t(T-S)\\t(T-S)^*&tD'},
    \]
    \[\text{where}\quad D\coloneqq\diag\paren[\big]{\sum_j(T_{i,j}-S_{i,j})}_i,D'\coloneqq\diag\paren[\big]{\sum_i(T_{i,j}-S_{i,j})}_j\]
    are diagonal matrices.
    Since the norm of this matrix is at most $4\sigma\leq8\sqrt{t}$, we have $\norm{T-S},\norm{D},\norm{D'}\leq8t^{-\frac12}$.
    By $\sum_jS_{i,j},\sum_iS_{i,j}\leq1+8t^{-\frac12}$ and $S_{i,j}\in t^{-1}\IN$, $\Width S\leq t(1+8t^{-\frac12})$ holds.
    Therefore, the proof ends by choosing $t>0$ to be $\ve=8t^{-\frac12}$ and $C>0$ to be $t(1+8t^{-\frac12})\leq C\ve^{-2}$.
\end{proof}

\begin{proof}[Proof of $\Roe\Univ=\IB_T$]
    Since a $d$-width operator $T\in\IB(\ell_2)_{\leq1}$ is $d$-$\ell_1$bounded, we have $\Roe\Univ\subset\IB_T$.
    Let $T\in\IB(\ell_2)$ be $d$-$\ell_1$bounded. To show that $T\in\Roe\Univ$, we may assume $T$ is doubly substochastic by scaling and decomposing each matrix entry into the positive and negative parts of its real and imaginary components. 

    Write $\ell_2$ as the increasing union $\ell_2=\overline{\bigcup_n\ell_2^n}$. 
    Define the projection $p_n$ onto $\ell_2^n\subset\ell_2$.
    Using the previous proposition for the doubly substochastic $p_nTp_n\in\IM(n,\IR)$, we get $\lfloor C\ve^{-2}\rfloor$-width $S_n\in\IM(n,\IR)$ such that $\norm{p_nTp_n-S_n}\leq\ve$. 
    Since $\norm{S_n}\leq\ve+\norm{T}$, we choose a WOT-accumulation point $S\in\IB(\ell_2)$ of $S_n\in\IB(\ell_2)$.
    It follows that $\Width S\leq C\ve^{-2}$ and $\norm{S-T}\leq\ve$.
    Since $\Roe\Univ$ is norm closed, we obtain $\IB_T\subset\Roe\Univ$.
\end{proof}

The Birkhoff--von Neumann theorem asserts that every doubly stochastic matrix of finite size is a convex combination of permutation matrices.
Birkhoff's Problem 111 is to extend this to the case of infinite size under a suitable topology.
See \cite{Isb62,Ken60,Gou24}.

\begin{thm}[Birkhoff's Problem 111 in the operator norm]
    There is a universal constant $C$ with the following property:
    For every $n\in\IN$, every doubly stochastic $T\in\IM(n,\IR)$, and every $\ve\in(0,1)$, there exist $m\leq C\ve^{-2}$ and permutation matrices $U_1,\dots, U_m$ such that $\norm{T-\frac1m\sum_iU_i}_{\op}\leq\ve$.
    The same holds true for a doubly substochastic infinite matrix $T$.
\end{thm}

\begin{proof}
    Let $T\in\IM(n,\IR)$ be a doubly stochastic matrix.
    The same argument as in Proposition~\ref{AI1} shows that there exists $S\in\IM(n,\IR)$ satisfying the following for $t>0$:
    \[\norm{T-S},\sup_i\abs*{\sum_j(T_{i,j}-S_{i,j})},\sup_j\abs*{\sum_i(T_{i,j}-S_{i,j})}\leq8t^{-\frac12}\quad\text{and }tS_{i,j}\in\IN.\]
    Set $m\coloneqq\lceil t+8\sqrt{t}\rceil$.
    Since $tS\in\IM(n,\IN)$ satisfies $\sum_jtS_{i,j},\sum_itS_{i,j}\in[t-8\sqrt{t},t+8\sqrt{t}]$, 
    there exists $B\in\IM(n,\IN)$ such that $S'\coloneqq m^{-1}(tS+B)$ is doubly stochastic.
    Indeed, for every $\mu,\nu\in\IN^n$ with $\inpro{\mu,\Iin}=\inpro{\nu,\Iin}$, there is $B\in\IM(n,\IN)$ such that $B\Iin=\mu,B^*\Iin=\nu$.
    By the Schur test, one has $\norm{B}_{\op}\leq16\sqrt{t}+1$ and
    \[\norm{T-S'}\leq\frac{t}{m}8t^{-\frac12}+\paren*{1-\frac{t}{m}}\norm{T}+m^{-1}(16\sqrt{t}+1)=O(t^{-\frac12}).\]
    Viewing $mS'$ as the $m$-regular bipartite multigraph on $[n]\sqcup[n]$, we apply the Hall marriage theorem to it repeatedly.
    Then, we obtain that $mS'$ is sum of permutations $U_1,\dots,U_m$ and the first assertion by choosing $t=O(\ve^{-2})$.
    
    Let $T\in\IB(\ell_2)$ be doubly substochastic.
    For $t>0$, the same argument as in Proposition~\ref{AI1} with a compactness argument shows that there exists $S\in\IB(\ell_2)$ with
    \[\norm{T-S}\leq8t^{-\frac12},\ \sup_i\sum_jS_{i,j},\sup_j\sum_iS_{i,j}\leq 1+8t^{-\frac12}\quad\text{and }tS_{i,j}\in\IN.\]
    We argue as for the first assertion but the fact $S\Iin$ and $S^*\Iin$ are not needed to be close to $\Iin$.
    We claim that there exists $B\in\IB(\ell_2)$ with nonnegative integer entries such that $\norm{B}_{\op}\leq O(\sqrt{t})$ and $S'\coloneqq m^{-1}(tS+B)$ is doubly stochastic for $m\coloneqq\lceil t+8\sqrt{t}\rceil+2$.
    Once we have proved the claim, we obtain our theorem by an $m$-edge-coloring of the $m$-regular bipartite multigraph $mS'$.
    To show the claim, it suffices to show that,
    for every $\mu\in(\IN_{\geq1})^\infty$ with $\sup_n\mu(n)<\infty$, there is $B_1\in\IB(\ell_2)$ with nonnegative integer entries such that $B_1\Iin=\mu,B_1^*\Iin=\Iin$, which is trivial.
    Indeed, applying this to $\mu\coloneqq (m-1)\Iin-tS\Iin,\nu\coloneqq(m-1)\Iin-tS^*\Iin$,
    we obtain $B_1,B_2$ such that $B_1\Iin=\mu,B_1^*\Iin=\Iin$ and $B_2\Iin=\nu,B_2^*\Iin=\Iin$.
    By the Schur test, one has $\norm{B_i}_{\op}\leq\sqrt{\norm{B_i\Iin}_\infty\norm{B_i^*\Iin}_\infty}=O(\sqrt{t})$ for $i=1,2$.
    $B\coloneqq B_1+B_2^*$ does the job.
\end{proof}

\subsection{Proof of $\Roe\Univ\subsetneq\Ql\Univ$}

Since the inclusion is clear, we show that this inclusion is strict. 
To show this, we decompose $\IN=\bigsqcup X_n$ and $\ell_2=\bigoplus\ell_2(X_n)$ with $\abs{X_n}\geq n$ and consider whether $\bigoplus p_n\in\prod\IM_{\abs{X_n}}$ belongs to $\Roe\Univ,\Ql\Univ$ or not for the Haar random projection $p_n\in\IM_{\abs{X_n}}$ with fixed rank $r_n\leq \frac{\abs{X_n}}{2}$.

Before the lemma, we comment on the $\ve$-covering number, that is, the minimal number of $\ve$-balls covering the space.
For the unit balls of $d$-dimensional normed spaces, this quantity is greater than $\ve^{-d}$ and less than $(1+\frac2\ve)^d$ \cite[Cor. 4.2.11]{Ver18}.

\begin{lem}\label{Ozawa}
    If $r_n=o(\abs{X_n})$, then $\bigoplus p_n\in\Ql\Univ$ almost surely.
\end{lem}

\begin{proof}
    This is essentially proved in \cite{Oza25}, so we only review the proof. 
    Fix expanders on $X_n$ with the Cheeger constant $\geq h>0$.
    Since $\min(\abs{A},\abs{B})/N\leq(1+h)^{-R/2}$ whenever $A,B\subset X_n$ satisfy $\operatorname{dist}(A,B)\geq R$, it suffices to show that for every $\ve>0$ there exists $\delta>0$ such that for every $E\subset X_n$ with $\abs{E}\leq\delta\abs{X_n}$ one has $\norm{\Iin_E p_n}\leq\ve$ with high probability.

    First, we have 
    \[\IP\left[x\in (\IC^N)_{=1}=S^{2N-1}\setcolon\norm{\Iin_Ex}\geq\sqrt{\frac{\abs{E}}{N}}+\ve\right]\leq\exp\paren*{-\frac{\ve^2}{2}\times2N},\quad\text{where } N=\abs{X_n}\]
    for $E\subset X_n$ by the Chernoff bound using the Gaussian, or a similar estimate by the measure concentration.
    %
    %
    Let $F\subset S^{2r_n-1}$ be a $\frac12$-dense subset with $\abs{F}\leq5^{2r_n}$ and we have $\norm{T}\leq2\max_{\xi\in F}\norm{T\xi}$ for any $T\in\IM(N\times r_n)$ \cite[Lem. 4.4.1]{Ver18}.
    %
    %
    Since $\abs{\set{E\subset X_n\setcolon\abs{E}\leq\delta N}}\leq\exp(H(\delta)N)$ with $H(\delta)\coloneqq-\delta\log\delta-(1-\delta)\log(1-\delta)$, for Haar random $U\in\Isom(N\times r_n,\IC)$, the union bound shows that
    \[\IP\left[U\in\Isom\setcolon\max_{\abs{E}\leq\delta N}\norm{\Iin_EU}\geq2(\sqrt{\delta}+\ve)\right]\leq e^{H(\delta)N}\times5^{2r_n}\times e^{-\ve^2N}.\]
    If we choose $\delta$ so that $H(\delta)<\ve^2$ and $r_n=o(N)$, this tail probability decays exponentially in $N$.
    The Borel--Cantelli lemma shows that there exists $n_0$ such that for every $n\geq n_0$ one has
    $\norm{\Iin_Ep_n}\leq2(\sqrt{\delta}+\ve)$ whenever $E\subset X_n$ is $\abs{E}\leq\delta\abs{X_n}$ almost surely.
    As in the proof of Theorem A, by replacing $R$ with $\max(R,\max_{n<n_0}\abs{X_n})$ and taking a countable intersection for $\ve$, we obtain $\bigoplus p_n\in\Ql\Univ$.
\end{proof}

\begin{lem}
    If $\frac{r_n}{\log\abs{X_n}}\to\infty$, then $\bigoplus p_n\notin\Roe\Univ$ almost surely.
\end{lem}

\begin{proof}
    Let $d>0$ and $\ve\in(0,\frac1{16})$.
    We show that $\sum_n\IP[p_n\text{ is $\ve$-$d$-width}]$ is finite.
    The key is that the $2\ve$-covering number of 
    \[E\coloneqq\set{T\in\IM_{\abs{X_n}}\setcolon\norm{T}\leq1,T\text{ is $\ve$-$d$-width}}\]
    is sufficiently smaller than that of the space of projections with rank $r_n$, which is denoted by $\Grass(\abs{X_n},r_n)$.
    For $\Sigma\subset X_n\times X_n$ with $\Width\Sigma\leq d$, the $(1+\ve)$-ball of the normed space $\set{T\in\IM_{\abs{X_n}}\mid\supp T\subset\Sigma}$ has the $\ve$-covering number at most $(1+\frac{2(1+\ve)}\ve)^{2\abs{\Sigma}}$.
    Since $E$ is covered by the union of $\ve$-neighborhoods of these $(1+\ve)$-balls and there are at most $(\abs{X_n}^d)^{\abs{X_n}}$ many such $\Sigma$, the $2\ve$-covering number of $E$ is at most
    \[M\coloneqq\paren*{1+\frac{2(1+\ve)}\ve}^{2d\abs{X_n}}\times\abs{X_n}^{d\abs{X_n}}.\]
    Thus, one has $\IP[E\cap\Grass(\abs{X_n},r_n)]\leq M\sup_{a\in\IM_{\abs{X_n}}}\IP[\Ball(a,2\ve)\cap\Grass(\abs{X_n},r_n)]$.

    We claim that 
    \[\IP[\Ball(a,2\ve)\cap\Grass(\abs{X_n},r_n)]\leq(16\ve)^{r_n(\abs{X_n}-r_n)}\]
    for $a\in\IM_{\abs{X_n}}$.
    We may assume that $\Ball(a,2\ve)\cap\Grass(\abs{X_n},r_n)\neq\emptyset$, and choose $p$ from the set.
    Note that $\Ball(a,2\ve)\cap\Grass(\abs{X_n},r_n)\subset\Ball(p,4\ve)\cap\Grass(\abs{X_n},r_n)$ and the measure of the latter set does not depend on $p$ by the rotation invariance of the Haar measure on $\Grass$.
    To prove the claim, it suffices to construct an $8\ve$-sparse subset $F'\subset\Grass(\abs{X_n},r_n)$ with $\abs{F'}\geq(16\ve)^{-r_n(\abs{X_n}-r_n)}$, that is, $d(x,y)\geq8\ve$ whenever $x,y\in F'$ are distinct.
    Define the contraction $\Phi:\Grass(\abs{X_n},r_n)\subset\IM_{\abs{X_n}}\to\IM(\abs{X_n}-r_n,r_n)$, given by the block $(2,1)$-entry.
    The image of $\Phi$ is $\IM(\abs{X_n}-r_n,r_n)_{\leq\frac12}$.
    Indeed, any pair of projections is unitary equivalent to the direct sum of 
    $\paren[\big]{\smallmat{1&0\\0&0},\smallmat{\cos\theta&\sin\theta}\smallmat{\cos\theta\\ \sin\theta}}$
    and $(0\text{or}1,0\text{or}1)$, so the image is the set of all matrices with the singular values $s_i=\sin\theta_i\cos\theta_i$.
    Fix a maximal $8\ve$-sparse subset $F\subset\IM(\abs{X_n}-r_n,r_n)_{\leq\frac12}$ and pullback to get $F'\subset\Grass(\abs{X_n},r_n)$ such that $\Phi(F')=F$ and $\abs{F'}=\abs{F}$.
    Thus, we have $\abs{F'}=\abs{F}\geq(16\ve)^{-r_n(\abs{X_n}-r_n)}$ by maximality, which proves the claim.
    One has
    \begin{align*}
        &\frac1{\abs{X_n}}\log\IP[p_n\text{ is $\ve$-$d$-width}]\\
        \leq&\log(16\ve)r_n\frac{\abs{X_n}-r_n}{\abs{X_n}}+d\log\abs{X_n}+\log\paren*{1+\frac{2(1+\ve)}{\ve}}2d\\
        \leq&-C_1r_n+C_2\log\abs{X_n}+C_3,\quad\text{where }C_1,C_2,C_3>0\\
        \to&-\infty\quad\text{as}\quad n\to\infty.
    \end{align*}
    In particular, $\frac1{\abs{X_n}}\log\IP[p_n\text{ is $\ve$-$d$-width}]\leq -c$ for some $c>0$ and sufficiently large $n$, so the Borel--Cantelli lemma shows that there exists $n_0$ such that $p_n$ is not $\ve$-$d$-width for every $n\geq n_0$ almost surely.
    Therefore, $\bigoplus p_n$ is not $\ve$-$d$-width almost surely, and by taking a countable intersection for $d$ (fix $\ve$), we obtain $\bigoplus p_n\notin\Roe\Univ$.
\end{proof}

Therefore, Haar random projections with rank $\log\abs{X_n}=o(r_n),r_n=o(\abs{X_n})$ distinguish $\Roe\Univ,\Ql\Univ$. 

\subsection{Proof of $\Ql\Univ\subsetneq Q$}\label{tensorJ}

\begin{proof}
    The inclusion is clear, and we show that the inclusion is strict.
    Fix a sequence of positive integers $\{m_n\}$ such that $m_n=o(n),m_n\to\infty$ and $m_n\geq1$.
    Set $X_n=[n]\times[m_n]$ and $p_n\coloneqq\frac1{m_n}\Iin\otimes\Iin^*$, where $\Iin$ is the constant vector $\Iin=(1,\dots,1)^T$.
    Choose a unitary $U_n\in\IM_n$ such that $\abs{U_n(i,j)}\leq O(\frac{1}{\sqrt{n}})$ (e.g., the discrete Fourier transform).
    To avoid confusion about indices, we only use ``$i,j$'' for elements in $[n]$ and ``$k,l$'' for elements in $[m_n]$.
    Hence, $(i,k),(j,l)\in X_n$ represent indices of matrices in $\IM_{\abs{X_n}}$

    First, we claim that $\bigoplus(U_n\otimes p_n)\in\prod\IM_{\abs{X_n}}$ belongs to $Q$.
    Indeed, for fixed $j,l$, 
    \[\sum_{(i,k)}\abs{(U_n\otimes p_n)_{(i,k),(j,l)}}^2=\sum_{(i,k)}\abs{U_n(i,j)}^2(m_n)^{-2}=\frac{\norm{U_n\delta_j}^2}{m_n}\leq\frac{\norm{U_n}^2}{m_n}\to0\]
    holds.
    Note that this convergence is uniform in $j,k$ and we have the same estimate for rows of $\bigoplus(U_n\otimes p_n)$.
    Fix $\ve>0$ and choose $n_0$ so that $m_n^{-\frac12}\leq\ve$ for $n\geq n_0$.
    On the tail blocks, every column and row is $\ve$-close to 0, while the finitely many initial blocks can be approximated by some finite-support vectors.

    Second, we claim that $\bigoplus(U_n\otimes p_n)\in\prod\IM_{\abs{X_n}}$ is not in $\Ql\Univ$.
    Suppose not, and fix $\Sigma\subset \IN\times \IN,\ve\in(0,1)$ such that $\Width\Sigma<\infty$ and $\norm{\Iin_A(U_n\otimes p_n)\Iin_B}\leq\ve$ whenever $A,B\subset X_n$ satisfy $(A\times B)\cap\Sigma=\emptyset$.
    Fix $j\in[n]$ and set $B\coloneqq\{j\}\times[m_n],$
    \[I\coloneqq\set{i\in[n]\setcolon(\{i\times[m_n])\times(j\times[m_n])\cap\Sigma=\emptyset}\]
    and $A\coloneqq I\times[m_n]$.
    Note that $\abs{I^c}\leq m_n\Width\Sigma$.
    Since $\Iin_A=\Iin_I\otimes\Iin,\Iin_B=\Iin_{\{j\}}\otimes\Iin$, one has
    \[\norm{\Iin_A(U_n\otimes p_n)\Iin_B}=\norm{(\Iin_IU_n\Iin_{j})\otimes p_n}=\sqrt{\sum_{i\in I}\abs{U_n(i,j)}^2}\geq\sqrt{1-m_nO\left(\frac1{n}\right)}\to1.\]
    This contradicts quasi-locality.
\end{proof}

\subsection{Proof of $M(\II)\subset Q$}

The following lemma states that every vector has either an essentially small support or many large level sets.

\begin{lem}
    Let $\ve>0,n\in\IN$ and $\xi\in\ell_2(\IN)$.
    Suppose that $\norm{\xi}\leq1$ and $\norm{\Iin_{F^c}\xi}\geq\ve$ for any $F\subset\IN$ with $\abs{F}\leq64\ve^{-4}n^2$.
    Then, there exist an index set $I$ and disjoint finite subsets $F_k\subset\IN,\ k\in I$ satisfying the following:
    \begin{itemize}
        \item $\abs{F_k}\geq n$,
        \item $\max_{x\in F_k}\abs{\xi(x)}\leq2\min_{x\in F_k}\abs{\xi(x)}\leq\frac{\ve^2}{2n}$,
        \item $\sum_{k\in I}\norm{\Iin_{F_k}\xi}^2\geq\frac{\ve^2}{2}$.
    \end{itemize}
\end{lem}

\begin{proof}
    By the layer-cake decomposition, we have $\sum_{k=1}^\infty 2^{-2k}\abs{F_k}\leq1$, where $F_k\coloneqq\set{x\setcolon 2^{-k}<\abs{\xi(x)}\leq2^{-k+1}}$.
    Define $\Bad\coloneqq\set{k\setcolon\abs{F_k}\leq\ve^22^{k-3}}$, then one has
    \[\sum_{k\in\Bad}\norm{\Iin_{F_k}\xi}^2\leq\sum_{k\in\Bad}2^{-2k+2}\abs{F_k}\leq\sum_k2^{-2k+2}\ve^22^{k-3}=\frac{\ve^2}{2}.\]
    Define $k_0\coloneqq\max\set{k\mid\ve^22^{k-3}\leq n}$, then one has
    \[\abs[\Big]{\bigcup_{k\leq k_0}F_k}\leq\sum_k2^{2k_0-2k}\abs{F_k}\leq2^{2k_0}\leq64\ve^{-4}n^2.\]
    Thus, $\sum_{k>k_0}\norm{\Iin_{F_k}\xi}^2\geq\ve^2$ holds.
    Define $I\coloneqq\set{k\mid k>k_0,\ k\notin\Bad}$, then one has $\sum_{k\in I}\norm{\Iin_{F_k}\xi}^2\geq\frac{\ve^2}{2}$.

    Since $k\notin\Bad$ for $k\in I$, $\abs{F_k}>\ve^22^{k-3}\geq\ve^22^{k_0-2}> n$ holds.
    One has $\max_{x\in F_k}\abs{\xi(x)}\leq2\min_{x\in F_k}\abs{\xi(x)}\leq2^{-k_0+1}<\frac{\ve^2}{2n}$.
\end{proof}

In particular, $\sum_k\frac{\abs{\inpro{\Iin_{F_k},\abs{\xi}}}^2}{\abs{F_k}}\geq\frac{\ve^2}{8}$ holds.

\begin{prop}
    If $T\in\IB(\ell_2)$ satisfies $T\II,\II T\subset Q$, then $T\in Q$.
    In particular, $M(\II)\subset Q$ holds.
    Moreover, $M(A\cap\mathrm{Ghost})\subset Q$ for any $\Cstar$-subalgebra $A\subset Q$ containing $\Roe\Univ$, where $\mathrm{Ghost}$ stands for the subspace of ghost operators.
\end{prop}

\begin{proof}
    Suppose that $T\in\IB(\ell_2)$ satisfies $T\II,\II T\subset Q$ but $T\notin Q$.
    By scaling, we may assume $\norm{T}\leq1$.
    Then there exists $\ve>0$ such that for every $d$, some column or row of $T$ is not $2\ve$-$d$-sparse.
    By replacing $T$ with $T^*$, we may assume that the columns are so.
    Choose distinct $\{x_n\}\subset\IN$ such that $\norm{\Iin_{F^c}T\delta_{x_n}}\geq2\ve$ for every $n$ and $F\subset\IN$ with $\abs{F}\leq64\ve^{-4}n^2$.
    Since $T\delta_{x_n}\rightharpoonup0$ in the weak topology of $\ell_2$, $\norm{\Iin_{F}T\delta_{x_n}}\to0$ for every $F\Subset\IN$ and we choose a subsequence $\{x_n'\}$ such that there exists a sequence of pairwise disjoint finite subsets $E_n'\Subset\IN$ satisfying $\norm{\Iin_{E_n'^c}T\delta_{x_n'}}\leq\ve$.
    
    Applying the previous lemma to $\xi_n'\coloneqq\Iin_{E_n'}T\delta_{x_n'}$ with parameter $\ve$, we obtain disjoint subsets $\{F_k^{n'}\}_{k,n'}$.
    Multiplying a diagonal operator on $T$ from the left, we may assume that $\xi_n'\geq0$ entrywise.
    Denote by $P_F\coloneqq\frac{1}{\abs{F}}\Iin_F\otimes\Iin_F^*$ the projection onto $\IC\Iin_F$.
    
    Since $P_F$ is $O(d^{-\frac12})$-$d$-width uniformly on $F$ by considering the adjacency matrix of expanders, $S\coloneqq\bigoplus_{k,n'}P_{F_k^{n'}}$ belongs to $\Roe\Univ$.
    Since $\abs{F_k^{n'}}\geq n'$ and there are finitely many $k$ for fixed $n'$ ($\supp\xi_n'$ is finite), $S$ is also in $\II$.
    Therefore, $ST\in Q$ and each $\eta_n'\coloneqq \Iin_{E'_n}ST\delta_{x_n'}=(\bigoplus_kP_{F_k^{n'}})\xi_n'$ is $\frac{\ve}{100}$-$d$-sparse for some $d$.
    However, by the previous lemma, one has $\norm{\eta_n'}\geq\frac{\ve}{2\sqrt2}$ and $\max_x\abs{\eta_n'(x)}\leq\frac{\ve^2}{2n'}$.
    This is a contradiction because for every $F\subset\IN$ with $\abs{F}\leq d$, we have $\norm{\Iin_F\eta_n'}\leq\sqrt{d}\max_x\abs{\eta_n'(x)}\to0$ as $n\to\infty$.
\end{proof}

Combining the preceding propositions, we obtain Theorem B.
\clearpage

\section{Proof of Theorem C}

This lemma comes from the original proof of Voiculescu's theorem. 

\begin{lem}
    Let $a,b\in\IB(\cH)$ be self-adjoint operators.
    There exist orthogonal subspaces $\cH_n\subset\cH\ (n=0,1,\dots)$ with $\dim\cH_n\leq2^n$ such that $a,b$ are tridiagonal with respect to this block decomposition, that is, $a\cH_n,b\cH_n\subset\cH_{n-1}\oplus\cH_n\oplus\cH_{n+1}$ for every $n$, and the compression $\Cstar(a,b)\subset\IB(\cH)\to\IB(\bigoplus_n\cH_n)$ is injective.
\end{lem}
\begin{proof}
    We may assume that $\cH$ is separable.
    There is a vector $\xi\in\cH$ such that $\Cstar(a,b)\subset\IB(\cH)\to\IB(\cH')$ is injective, where $\cH'\coloneqq\overline{\Cstar(a,b)\xi}$ is the subrepresentation generated by $\xi$.
    Indeed, decompose $\cH$ into the direct sum of cyclic representations and take $\xi$ as the weighted sum of these cyclic vectors.
    %
    %

    Define subspaces 
    \[\cH_0\coloneqq\IC\xi,\ \cH_{n+1}\coloneqq\paren[\big]{(\bigoplus_{k\leq n}\cH_k)+a\cH_n+b\cH_n}\cap\paren[\big]{\bigoplus_{k\leq n}\cH_k}^\perp,\]
    inductively.
    By $a\cH_n,b\cH_n\subset\cH_{n+1}\oplus\paren[\big]{\bigoplus_{k\leq n}\cH_k}$, $\bigoplus_n\cH_n$ is a $\Cstar(a,b)$-invariant subspace and $=\cH'$.
    $a\cH_n,b\cH_n$ are orthogonal to $\cH_m$ for $m\geq n+2$ by the definition and so are for $m\leq n-2$ by self-adjointness.
    Note that $\dim\cH_{n+1}\leq\dim(a\cH_n+b\cH_n)\leq2\dim\cH_n$.
\end{proof}

\begin{prop}
    There is a bounded degree infinite graph $X$ such that every separable $\Cstar$-algebra embeds into $\Ql X\ (\subset\Ql\Univ)$.
    The image of this embedding is ghost.
\end{prop}

\begin{proof}
    The graph is constructed on the vertex set $X=\bigsqcup_n X_n$, $\abs{X_n}=3^n$ with $d$-regular expanders on $X_n\sqcup X_{n+1}$ for each $n$.
    Thus, the resultant graph $X$ has degree $2d$.

    Let $A$ be a separable $\Cstar$-algebra.
    Tensoring with $\IK$, we may assume that $A$ is generated by two self-adjoint $a,b$ \cite{OZ76}.
    Take a representation of $A$ on $\cH=\bigoplus_n\cH_n$
    with $\dim\cH_n\leq2^n$ such that $a,b$ are tridiagonal on this block decomposition.

    Define the Haar random isometry $V_n\colon\cH_n\to\ell_2(X_n)$ and $V=\bigoplus V_n\colon\cH\to\ell_2X$.
    We claim that the $*$-homomorphism $\pi\colon\IB(\cH)\ni T\mapsto VTV^*\in\IB(\ell_2X)$ maps every block-tridiagonal operator into $\Ql X$.
    Once the claim is proved, we obtain $\pi(a),\pi(b)\in\Ql X$, and hence $\pi(A)\subset\Ql X$.

    Let us prove the claim.
    By Lemma~\ref{Ozawa} or \cite{Oza25}, $\norm{\Iin_E V_n}\leq2(\sqrt{\delta}+\ve)$ whenever $E\subset X_n$ satisfies $\abs{E}\leq\delta\abs{X_n}$ almost surely, for $\ve,\delta$ with $H(\delta)<\ve^2$.
    Note that $V_n\oplus V_{n+1}\colon\cH_n\oplus\cH_{n+1}\to\ell_2(X_n\sqcup X_{n+1})$ is not the Haar random isometry $\cH_n\oplus\cH_{n+1}\to\ell_2(X_n\sqcup X_{n+1})$ but we have a similar estimate for $V_n\oplus V_{n+1}$:
    \[\norm{\Iin_E(V_n\oplus V_{n+1})}\leq2(\sqrt{\delta}+\ve)\text{ whenever }E\subset X_n\sqcup X_{n+1}\text{ is }\abs{E}\leq\frac{\delta}{4}\abs{X_n\sqcup X_{n+1}}.\]
    Arguing as in Lemma~\ref{Ozawa} or \cite{Oza25} for the expanders on $X_n\sqcup X_{n+1}$ and $\dim(\cH_n\oplus\cH_{n+1})=o(\abs{X_n\sqcup X_{n+1}})$, we have that $\pi$ maps $\prod_{n:\text{even}}\IB(\cH_n\oplus\cH_{n+1}),\prod_{n:\text{odd}}\IB(\cH_n\oplus\cH_{n+1})$ into $\Ql X$.
    Since every block-tridiagonal matrix is a sum of $\prod_{\text{odd}}$ and $\prod_{\text{even}}$, we obtain the claim.
\end{proof}

The following lemma shows that the previous proposition is genuinely nontrivial.
This lemma is the universal version of \cite[Lem. 1.11]{BB+24}

\begin{lem}\label{BCT}
    $\IB(\cH)$ does not embed into $\Ql\Univ$ if $\dim\cH=\infty$.
\end{lem}

\begin{proof}
    Suppose not and fix a nonzero $*$-homomorphism $\pi\colon\IB(\cH)\to\Ql\Univ\subset\IB(\ell_2)$.
    This is automatically normal by the separability of $\ell_2$.
    Indeed, we need to show that there is no nonzero representation of $\IB/\IK$ on $\ell_2$, and this follows from the fact that there are continuum-many orthogonal projections and the simplicity of $\IB/\IK$ (See \cite{Tak68}).

    Fix $0<\ve<\frac14$.
    The unit ball of $\Ql\Univ$ is covered by the WOT-closed sets for $d$
    \[\mathrm{QL}(\ve,d)\coloneqq\set{T\in\IB(\ell_2)\setcolon\norm{T}\leq1,\ T\text{ is $\ve$-$d$-quasi-local}}\]
    as shown in \cite[Lem. 4.2]{BB+24}.
    By the Baire category theorem, there exists $d\in\IN$ such that $\pi^{-1}(\mathrm{QL}(\ve,d))\subset\IB(\cH)_{\leq1}$ has a WOT-interior point.
    Hence, there exist a finite-rank projection $p\in\IB(\cH)$ and $T_0\in\IB(p\cH)_{\leq1}$ such that $T_0+T\in\pi^{-1}(\mathrm{QL}(\ve,d))$ for any $T\in\IB((1-p)\cH)_{\leq1}$.
    By $2\pi(T)\in\pi^{-1}(\mathrm{QL}(\ve,d)-\mathrm{QL}(\ve,d))$, we have $\pi(T)\in\mathrm{QL}(\ve,2d)$ for any $T\in\IB((1-p)\cH)_{\leq1}$.
    Restricting $\pi$ to $\IB((1-p)\cH)$, we may assume that $\pi\colon\IB(\cH)\to\Ql\Univ$ maps the unit ball into $\mathrm{QL}(\ve,2d)$.

    The normal $*$-homomorphism $\pi$ is represented as $\pi(a)=\sum_i V_iaV_i^*$, where $V_i\colon\cH\to\ell_2$ are isometries with orthogonal ranges.
    Fix $e\in\cH_{=1}$ and $V=V_1$.
    Choose a finite-supported vector $\xi\in\ell_2$ such that $\norm{\xi-Ve}\leq\ve,\norm{\xi}\leq1$.
    Since $\pi(\theta\otimes e^*)\in\mathrm{QL}(\ve,2d)$ for $\theta\in\cH_{=1}$, we have $F=F_\theta\subset\IN$ such that $\abs{F}\leq 2d\abs{\supp\xi}$ and $\norm{\Iin_{F^c}\pi(\theta\otimes e^*)\xi}\leq\ve$.
    We have $\pi(\theta\otimes e^*)Ve=V\theta$ because of orthogonal ranges, so $\norm{\Iin_{F^c}V\theta}\leq2\ve$ holds for every $x\in\IN$.
    Fix an orthonormal system $\set{\theta_n}_n\subset\cH$.
    Since $\theta_n\rightharpoonup0$ in the weak topology of $\cH$,
    by replacing $\set{\theta_n}$ with its subsequence,
    we may assume that there exists a sequence of pairwise disjoint finite subsets $E_n\Subset\IN$ such that $\norm{\Iin_{E_n^c}V\theta_n}\leq2^{-n}$.
    In particular, $\sum_{n=1}^\infty\abs{\inpro{\delta_x,V\theta_n}}\leq2$ holds.
    Set $\theta\coloneqq\frac1{\sqrt{m}}\sum_{n=1}^m\theta_n$ for sufficiently large $m$.
    By $\sup_x\abs{\inpro{\delta_x,V\theta}}\leq\frac{2}{\sqrt{m}}$, we obtain $\norm{\Iin_F V\theta}\leq2\sqrt{\frac{\abs{F}}{m}}$ for $F\Subset\IN$, which contradicts sparseness of $V\theta$ for sufficiently large $m$.
\end{proof}

\begin{rem}
    $\IB(\ell_2)$ embeds into $Q$ via the isometry $V\colon\ell_2\to\bigoplus_n\ell_2^n$ mapping the canonical basis $e_n$ to the $n$-th direct summand $\frac1{\sqrt{n}}\Iin\in\ell_2^n$ by the same argument as in Subsection~\ref{tensorJ}.
\end{rem}

\begin{lem}[See \cite{BGP26}]
    For a separable $\Cstar$-algebra $A$, $A$ embeds into $\Roe\Univ$ if and only if $A$ embeds into $\mathrm{C}^*_\text{r}G$ for some \'{e}tale locally compact second countable groupoid $G$.
\end{lem}

\begin{proof}
    First suppose that $A$ embeds into $\mathrm{C}^*_\text{r}G$.
    Since $A\subset\mathrm{C}^*_\text{r}G$ is represented on $\bigoplus_{x\in G^{(0)}}\ell_2(G_x)$, we fix a countable set $F\subset G^{(0)}$ such that $A\acts\bigoplus_{x\in F}\ell_2(G_x)$ is faithful.
    $G_x$ is countable by discreteness and the images of elements of $C_c(G)\subset \mathrm{C}^*_\text{r}G\subset\IB(\ell_2(\bigsqcup_{x\in F}G_x))$ are bounded-width operators by compactness of their supports.

    Conversely, suppose that $A\subset\Roe\Univ$.
    We may assume that $\IK\subset A$.
    Every $d$-width operator is a sum of at most $2d-1$ $1$-width operators by the edge-coloring.
    Every $1$-width operator is a linear combination of at most $O(\ve^{-1})$ partial translations (that is, $1$-width operators with $0,1$-entries) up to $\ve$ in the operator norm.
    Replacing $A$ with a larger one, we may assume $A$ is generated by countably many partial translations $s_i\colon\mathrm{dom}(s_i)\to\mathrm{cod}(s_i)$.
    Moreover, we may assume that all fixed point projections of products of $\{s_i\},\{s_i^*\}$ belong to $A\cap\ell_\infty$.
    Hence, $\ell_\infty\cap A\ (\supset c_0)$ is the Cartan subalgebra of $A$ and $A$ is the reduced $\Cstar$-algebra of the groupoid generated by $s_i$ on the Gelfand dual of $\ell_\infty\cap A$ \cite{Ren08}.
    Equivalently, $G$ can be described explicitly as the translation groupoid of \cite{Roe03}.
    In this construction the twist is trivial.
\end{proof}

\setcounter{thmA}{0}
\begin{corA}
    Let $A$ be a separable quasidiagonal $\Cstar$-algebra.
    Then $A$ embeds into $\Roe\Univ$ and $\mathrm{C}^*_{\text{r}}G$ for some \'{e}tale locally compact second countable groupoid $G$.
\end{corA}

\begin{proof}
    This is a corollary of Theorem A.
    Passing to a sufficiently rapidly increasing subsequence of the quasidiagonal projections,
    we obtain a representation $\pi\colon A\hookrightarrow\IB(\bigoplus_i\ell_2^{n_i})$ such that the image of $A$ is contained in $\IK+\prod_i\IM_{n_i}$.
    The embedding $\prod_i\IM_{n_i}\to\Roe\Univ$ is implemented by $\Ad V:\IB(\bigoplus_i\ell_2^{n_i})\to\IB(\ell_2)$.
    This maps $\IK(\bigoplus_i\ell_2^{n_i})$ into $\IK\subset\Roe\Univ$, so this also maps $\pi(A)$ into $\Roe\Univ$.
\end{proof}

\end{document}